\documentclass[a4paper]{article}
\usepackage[utf8]{inputenc}
\usepackage[english]{babel}
\usepackage{paralist,url,verbatim,anysize}
\usepackage{amscd}
\usepackage{color}
\usepackage[usenames,dvipsnames,svgnames,table]{xcolor}
\definecolor{purple}{HTML}{961C8C}
\usepackage{setspace}
\usepackage{authblk}
\usepackage{mathrsfs}
\usepackage{xparse}
\usepackage{suffix}
\usepackage{microtype}
\usepackage{manfnt}
\usepackage{nicefrac}
\usepackage[e]{esvect}
\usepackage{amsmath,amsthm,amssymb,amsfonts}
\usepackage{graphicx,graphics}
\usepackage[bookmarksopen=true, backref=page]{hyperref}
\hypersetup{colorlinks=false, citecolor=blue, urlcolor=purple}
\usepackage{epstopdf}
\usepackage[mathlines, pagewise]{lineno}
\usepackage{bbm}
\usepackage[font=small,labelfont=bf]{caption}
\usepackage[labelfont=rm, textfont=rm, labelformat=simple]{subcaption}
\theoremstyle{plain}
\newtheorem{theorem}{\bf Theorem}[section]

\newtheorem*{theorem*}{Theorem}
\newtheorem*{conjecture*}{Conjecture}
\newtheorem*{problem*}{Problem}
\newtheorem{cor}[theorem]{Corollary}
\newtheorem{lemma}[theorem]{Lemma}
\newtheorem{prp}[theorem]{Proposition}
\newtheorem{mthm}{\bf Main Theorem}

\newtheorem{mcor}[mthm]{\bf Corollary}

\newtheorem{problem}[theorem]{Problem}
\theoremstyle{definition}
\newtheorem{rem}[theorem]{Remark}
\newtheorem{definition}[theorem]{Definition}

\newtheorem{example}[theorem]{Example}

\DeclareFontFamily{OT1}{pzc}{}
\DeclareFontShape{OT1}{pzc}{m}{it}{<-> s * [1.2] pzcmi7t}{}
\DeclareMathAlphabet{\mathpzc}{OT1}{pzc}{m}{it}
\newenvironment{frcseries}{\fontfamily{frc}\selectfont}{}

\newcommand{\ii}[1]{\iota(#1)}
\newcommand\Defn[1]{\emph{#1}}
\newcommand{\comp}{\operatorname{c}}
\newcommand{\RS}{\operatorname{R}}
\newcommand{\R}{\mathbb{R}}
\newcommand{\CC}{\mathbb{C}}

\newcommand{\Lk}{\mathrm{Lk}}
\newcommand{\F}{\mathrm{F}}
\newcommand{\CF}{\mathrm{C}}
\newcommand{\sd}{\mathrm{sd}\, }
\newcommand{\SSp}{\mathrm{sp}^s}
\newcommand{\Sp}{\mathrm{sp}}
\newcommand{\io}{\mathrm{(A)}}
\newcommand{\iit}{\mathrm{(B)}}
\newcommand{\TT}{\mathrm{T}}
\newcommand{\intx}{\mathrm{int}\,}
\newcommand{\RN}{\mathrm{T}}
\newcommand{\rint}{\mathrm{relint}\,}

\newcommand{\HA}{\mathscr{A}}
\newcommand{\EH}{\HA_e}
\newcommand{\SH}{\HA_\sigma}
\newcommand{\K}{\mathscr{K}}
\newcommand{\s}{\mathbf{s}}
\newcommand{\cc}{2}
\newcommand{\OD}{O}
\newcommand{\FD}{F}
\begin{document}
\setlength{\affilsep}{1mm}

\author{Karim A. Adiprasito\thanks{Supported by an EPDI/IPDE fellowship, the DFG within the research training group ``Methods for Discrete Structures'' (GRK1408) and by the Romanian NASR, CNCS -- UEFISCDI, project PN-II-ID-PCE-2011-3-0533.}}
\affil{\footnotesize Institut des Hautes \'Etudes Scientifiques\\
\footnotesize Le Bois Marie 35, Route de Chartres\\
\footnotesize 91440 Bures-sur-Yvette, France\\
\footnotesize \url{adiprasito@math.fu-berlin.de}, \url{adiprasito@ihes.fr}\\
\mbox{}\\
\footnotesize Received March 7, 2013, Revised June 22, 2014}
\date{\vspace{-5ex}}
\title{Combinatorial stratifications and minimality of 2-arrangements}
\maketitle

\begin{abstract}
We prove that the complement of any affine $2$-arrangement in $\R^d$ is 
\emph{minimal}, that is, it is homotopy equivalent to a cell complex with as 
many $i$-cells as its $i$-th rational Betti number. For the proof, we provide a 
Lefschetz-type hyperplane theorem for complements of 
2-arrangements, and introduce Alexander duality for combinatorial Morse functions. 
Our results greatly generalize previous work by Falk, Dimca--Papadima, Hattori, 
Randell, and Salvetti--Settepanella and others, and they demonstrate that 
in contrast to previous investigations, a purely combinatorial approach suffices 
to show minimality and the Lefschetz Hyperplane Theorem for complements of 
complex hyperplane arrangements.
\end{abstract}

\section{Introduction}

A $c$-arrange\-ment is a finite collection of distinct affine subspaces of $\R^d$, all of codimension $c$, with the property that the codimension of the non-empty intersection of any subset of $\HA$ is a multiple of $c$. For example, after identifying $\CC$ with $\R^2$, any collection of hyperplanes in $\CC^d$ can be viewed as a $2$-arrangement in $\R^{2d}$. However, not all $2$-arrangements arise this way, cf.~\cite[Sec.\ III,\ 5.2]{GM-SMT},~\cite{Z-DRC}. 
In this paper, we study the complement $\HA^{\comp}:=\R^{d}{\setminus} \HA$ of any $2$-arrangement $\HA$ in $\R^d$.

Subspace arrangements $\HA$ and their complements $\HA^{\comp}$ have been extensively studied in several areas of mathematics. Thanks to the work of Goresky and MacPherson~\cite{GM-SMT}, the homology of $\HA^{\comp}$ is well understood; it is determined by the \Defn{intersection poset} of the arrangement, which is the set of all nonempty intersections of its elements, ordered by reverse inclusion. In fact, the intersection poset determines even the homotopy type of the compactification of $\HA$~\cite{ZieZiv}. On the other hand, it does not determine the homotopy type of the complement of $\HA^{\comp}$, even if we restrict ourselves to complex hyperplane arrangements~\cite{Bartolo, Rybnikovold, Rybnikov}, and understanding the homotopy type of $\HA^{\comp}$~remains~challenging.

A standard approach to study the homotopy type of a topological space $X$ is to find a \Defn{model} for it, that is, a CW complex homotopy equivalent to it. By basic cellular homology any model of a space $X$ has at least $\beta_i(X)$ $i$-cells for each $i$, where $\beta_i$ is the $i$-th (rational) Betti number. A natural question arises: Is the complement of an arrangement \Defn{minimal}, i.e., does it have a model with \textit{exactly} $\beta_i(X)$ $i$-cells for all $i$?

Building on previous work by Hattori~\cite{Hattori}, Falk~\cite{Falk} and Cohen--Suciu~\cite{CohenSuciu}, around 2000 Dimca--Papadima~\cite{DimcaPapadima} and Randell~\cite{Randell} independently showed that the complement of any complex hyperplane arrangement is a minimal space. The idea is to establish a Lefschetz-type hyperplane theorem for the complement of the arrangement by first establishing a Lefschetz hyperplane theorem for the Milnor fiber of the arrangement, building on earlier work of Hamm and L{\^e} \cite{Hamm, HammLe}. An elegant inductive argument completes their proof. 

On the other hand, the complement of an arbitrary subspace arrangement is, in general, \emph{not} minimal. In fact, complements of subspace arrangements might have arbitrary torsion in cohomology (cf.~\cite[Sec.\ III, Thm.\ A]{GM-SMT}). This naturally leads to the following question:
\begin{problem}[Minimality]\label{prb:mini}
Is the complement $\HA^{\comp}$ of an arbitrary $c$-arrangement $\HA$ minimal?
\end{problem}

The interesting case is $c=2$. In fact, if $c$ is not $2$, the complements of $c$-arrangements, and even $c$-arrangements of pseudospheres (cf.~\cite[Sec.\ 8 \& 9]{BjZie}), are easily shown to be minimal; see Section~\ref{sec:car}. 
In 2007, Salvetti--Settepanella~\cite{SalSet} proposed a combinatorial approach to Problem~\ref{prb:mini}, based on Forman's discretization of Morse theory~\cite{FormanADV}. Discrete Morse functions are defined on regular CW complexes rather than on manifolds; instead of critical points, they have combinatorially-defined \Defn{critical faces}. Any discrete Morse function with $c_i$ critical $i$-faces on a complex $C$ yields a model for $C$ with exactly $c_i$ $i$-cells (cf.\ Theorem~\ref{thm:MorseThmw}). 
Salvetti--Settepanella studied discrete Morse functions on the \Defn{Salvetti complexes}~\cite{Salvetti}, which are models for complements of complexified real arrangements. Remarkably, they found that all Salvetti complexes admit \Defn{perfect} discrete Morse functions, that is, functions with exactly $\beta_i(\HA^{\comp})$ critical $i$-faces. Formans's Theorem~\ref{thm:MorseThmw} now yields the desired minimal models for $\HA^{\comp}$.
 
This tactic does not directly extend to the generality of complex hyperplane arrangements. However, models for complex arrangements, and even for $c$-arrangements, have been introduced and studied by Bj\"orner and Ziegler~\cite{BjZie}. In the case of complexified-real arrangements, their models contain the Salvetti complex as a special case. While our notion of the combinatorial stratification is slightly more restrictive than Bj\"orner--Ziegler's, cf.\ Section~\ref{ssc:2-arrangements}, it still includes most of the combinatorial stratifications studied in~\cite{BjZie}. For example, we still recover the $\s^{(1)}$-stratification which gives rise to the Salvetti complex. With these tools at hand, we can tackle Problem~\ref{prb:mini} combinatorially:
\enlargethispage{3mm}
\begin{problem}[Optimality of classical models]\label{prb:cmpt}
Are there perfect discrete Morse functions on the Bj\"orner--Ziegler models for the complements of arbitrary $\cc$-arrangements?
\end{problem}

We are motivated by the fact that discrete Morse theory provides a simple yet powerful tool to study stratified spaces. On the other hand, there are several difficulties to overcome. In fact, Problem~\ref{prb:cmpt} is more ambitious than Problem~\ref{prb:mini} in many respects:

\begin{compactitem}[$\circ$]
\item Few regular CW complexes, even among the minimal ones, admit perfect discrete Morse functions. For example, many $3$-balls~\cite{Bing} and many contractible $2$-complexes~\cite{Zeeman} are not collapsible.
\item There are few results in the literature predicting the existence of perfect Morse functions. For example, it is not known whether any subdivision of the $4$-simplex is collapsible, cf.~\cite[Prb.\ 5.5]{Kirby}.
\item Solving Problem~\ref{prb:cmpt} could help in obtaining a more explicit picture of the attaching maps for the minimal model; compare Salvetti--Settepanella~\cite{SalSet} and Yoshinaga~\cite{Yoshi}.
\end{compactitem}

\noindent In this paper, we answer both problems in the affirmative.

\begin{mthm}[Theorem~\ref{thm:BZperfect}]\label{MTHM:BZP}
Any complement complex of any $\cc$-arrange\-ment $\HA$ in $S^d$ or~$\R^d$ admits a perfect discrete Morse function.
\end{mthm}

\begin{mcor}\label{mcor:m}
The complement of any affine $\cc$-arrangement in $\R^d$, and the complement of any $\cc$-arrangement in $S^d$, is a minimal space. 
\end{mcor}

A crucial step on the way to the proof of Theorem~\ref{MTHM:BZP} is the proof of a Lefschetz-type hyperplane theorem for the complements of $\cc$-arrangements. The lemma we actually need is a bit technical (Corollary~\ref{cor:lef}), but its topological core can be phrased in the following way:

\begin{mthm}[Theorem~\ref{thm:lef}] \label{MTHM:LEFT}
Let $\HA^{\comp}$ denote the complement of any affine $\cc$-arrangement $\HA$ in $\R^d$, and let $H$ be any hyperplane in $\R^d$ in general position with respect to $\HA$. Then $\HA^{\comp}$ is homotopy equivalent to $H\cap\HA^{\comp}$ with finitely many $e$-cells attached, where $e = \lceil\nicefrac{d}{\cc}\rceil = d-\lfloor\nicefrac{d}{\cc}\rfloor$.
\end{mthm}

An analogous theorem holds for complements of $c$-arrangements ($c\neq 2$, with $e = d-\lfloor\nicefrac{d}{c}\rfloor$); it is an immediate consequence of the analogue of Corollary~\ref{mcor:m} for $c$-arrangements, $c\neq 2$, cf.\ Section~\ref{sec:car}.
Theorem~\ref{MTHM:LEFT} extends a result on complex hyperplane arrangements, which follows from Morse theory applied to the Milnor fiber~\cite{DimcaPapadima,
HammLe, Randell}. The main ingredients to our study are:

\begin{compactitem}[$\circ$]
\item the formula to compute the homology of subspace arrangements in terms of the intersection lattice, due to Goresky and MacPherson~\cite{GM-SMT}; cf.\   Lemma~\ref{LEM:LHTCA};
\item the study of combinatorial stratifications as initiated by Bj\"orner and Ziegler~\cite{BjZie}; cf.\ Section~\ref{ssc:2-arrangements};
\item the study of the collapsibility of complexes whose geometric realizations satisfy certain geometric constraints, as discussed previous work of Benedetti and the author, cf.~\cite{AB-SSZ}; this is for example used in the proof of Theorem~\ref{thm:hemisphere};
\item the idea of Alexander duality for Morse functions, in particular the elementary notion of ``out-$j$ collapse'', introduced in Section~\ref{ssc:relcollapse};
\item the notion of (Poincar\'e) duality of discrete Morse functions, which goes back to Forman~\cite{FormanADV}. This is used to establish discrete Morse functions on complement complexes, cf.\ Theorem~\ref{ssc:cmpm}.
\end{compactitem}

\section{Preliminaries}\label{sec:prelim}
We use this section to recall the basic facts on discrete Morse theory, $\cc$-arrangements and combinatorial stratifications, and to introduce some concepts that we shall use for the proofs of the main results of this paper. A notion central to (almost) every formulation of Lefschetz-type hyperplane theorems is that of \Defn{general position}. In our setting, we can make this very precise. Recall that a \Defn{polyhedron} in $S^d$ (resp.\ $\R^d$) is an intersection of closed hemispheres (resp.\ halfspaces).

\begin{definition} If $H$ is a hyperplane in $S^d$ (resp.\ $\R^d$), then $H$ is in \Defn{general position} with respect to a polyhedron if $H$ intersects the span (resp.\ affine span) of any face of the polyhedron transversally. This notion extends naturally to collections of polyhedra (such as for instance polytopal complexes or subspace arrangements): A hyperplane is in \Defn{general position} to such a collection if it is in general position with respect to all of its elements, and every intersection of its elements. We say that a hemisphere is in \Defn{general position} with respect to a collection of polyhedra if the boundary of the hemisphere is in general position with respect to the collection.
\end{definition}

To prove the main results, we work with arrangements in $S^d$, and treat the euclidean case as a special case of the spherical one. The sphere $S^d$ shall always be considered as the unit sphere in $\R^{d+1}$ (with midpoint at the origin). An \Defn{$i$-dimensional subspace in $S^d$} is the intersection of the unit sphere $S^d$ with some $(i+1)$-dimensional linear subspace of $\R^{d+1}$.  A \Defn{hyperplane in $S^d$} is a $(d-1)$-dimensional subspace in $S^d$. We use $\Sp(X)$ to denote the \Defn{linear span} of a set $X$ in~$\R^d$, and if $X$ is a subset of $S^d\subset \R^{d+1}$, we define $\SSp(X)=\Sp(X)\cap S^d$, the \Defn{spherical span} of $X$ in $S^d$. We use $\intx X$, $\rint X$ and $\partial X$ to denote the \Defn{interior}, \Defn{relative interior} and \Defn{boundary} of a set $X$ respectively. We will frequently abuse notation and treat a subspace arrangement both a collection of subspaces and the union of its elements; for instance, we will write $\R^d{\setminus} \HA$ to denote the complement of an arrangement $\HA$ in $\R^d$. 

\subsection{Discrete Morse theory and duality, I}\label{sec:dmt}

We shall use this section to recall the main terminology for discrete Morse theory used here. For more information on discrete Morse theory, we refer the reader to Forman~\cite{FormanADV}, Chari~\cite{Chari} and Benedetti~\cite{B-DMT4MWB}. For \Defn{CW complexes} and \Defn{regular CW complexes} we refer the reader to Munkres~\cite[Ch.\ 4,\ \S 38]{Munkres}. Roughly speaking, a regular CW complex is a collection of open balls, the \Defn{cells}, identified along homeomorphisms of their boundaries. The closures of the cells of a regular CW complex $C$ are the \Defn{faces} of the complex, the union of which we denote by~$\F(C)$. The common terminology used for simplicial and polytopal complexes (\Defn{facet}, \Defn{dimension}) naturally extends to regular CW complexes, we refer to \cite{Grunbaum, RourkeSanders, Z} for the basic notions. For the notion of \Defn{dual block complex}, we also refer to Munkres~\cite[Ch.\ 8,\ \S 64]{Munkres}, and Bj\"orner et al.~\cite[Prp.\ 4.7.26(iv)]{BLSWZ}.

A  \Defn{discrete vector field} $\varPhi$ on a regular CW complex $C$ is a collection of pairs~$(\sigma,\varSigma)$ of nonempty faces of $C$, the \Defn{matching pairs} of $\varPhi$, such that $\sigma$ is a codimension-one face of $\varSigma$, and no face of $C$ belongs to two different pairs of $\varPhi$. If $(\sigma,\varSigma)\in \varPhi$, we also say that $\sigma$ \Defn{is matched with} $\varSigma$ in $\varPhi$. 
A \Defn{gradient path} in $\varPhi$ is a union of pairs in $\varPhi$
\[ (\sigma_0, \varSigma_0), (\sigma_1, \varSigma_1),  \ldots, (\sigma_k, \varSigma_k),\ k\ge 1,\]
such that $\sigma_{i+1}\neq \sigma_i$ and $\sigma_{i+1}$ is a codimension-one face of $\varSigma_i$ for all $i\in\{0,1,\, \cdots, k-1\}$; the gradient path is \Defn{closed} if $\sigma_0 = \sigma_k$. A discrete vector field $\varPhi$ is a \Defn{Morse matching} if $\varPhi$ contains no closed gradient paths. Morse matchings, a notion due to Chari~\cite{Chari}, are only one of several ways to represent \Defn{discrete Morse functions}, which were originally introduced by Forman~\cite{FormanADV, FormanUSER}. They can be thought of as the gradient flow associated to a Morse function.

Discrete Morse theory is a generalization of Whitehead's notion of collapsibility \cite{Whitehead}: Recall that a regular CW complex $C$ \Defn{elementary collapses} to a subcomplex $C'$ if $C'$ is obtained from $C$ by deleting a nonempty face $\sigma$ that is strictly contained in only one other face of $C$. We say $C$ \Defn{collapses} to $C'$, and write $C\searrow C'$, if $C$ can be deformed to $C'$ by elementary collapsing steps; $C$ is \Defn{collapsible} if it collapses to some vertex. A collapse provides a certificate for homotopy equivalence, hence, every collapsible complex is contractible.

It is an easy exercise to show that a regular CW complex $C$ is collapsible if and only if it admits a Morse matching with the property that all but one face are in one of the matching pairs. 
Forman generalized this observation with the following notion: If $\varPhi$ is a Morse matching on a regular CW complex $C$, then $\CF(\varPhi)$ is the set of \Defn{critical faces}, that is, $\CF(\varPhi)$ is the set of faces of $C$ that are not in any of the matching pairs of $\varPhi$. We set $\CF_i(\varPhi)$ to denote the subset of elements of $\CF(\varPhi)$ of dimension $i$, and we use $c_i(\varPhi)$ to denote the cardinality of this set. Let $A \simeq B$ indicate the homotopy equivalence of $A$ and $B$.

\begin{theorem}[Forman {\cite[Cor.\ 3.5]{FormanADV}}, {\cite[Thm.\ 3.1]{Chari}}]\label{thm:MorseThmw}
Let $C$ be a regular CW complex. Given any Morse matching $\varPhi$ on $C$, we have $C\simeq \Sigma(\varPhi)$, where $\Sigma(\varPhi)$ is a CW complex whose $i$-cells are in natural bijection with the critical $i$-faces of the Morse matching $\varPhi$.
\end{theorem}

Theorem~\ref{thm:MorseThmw} is a special case of a more powerful result of Forman, Theorem~\ref{thm:MorseThm}. To state his theorem in a form convenient to our research, however, we need some more notation; we return to this in Section~\ref{sec:morsethm}. The notion of Morse matchings with critical faces subsumes Whitehead's notion of collapses. 
\begin{compactitem}[$\circ$]
\item A regular CW complex $C$ is collapsible if and only if it admits a Morse matching with only one critical~face. 
\item More generally, a regular CW complex $C$ collapses to a subcomplex $C'$ if and only if $C$ admits a Morse matching $\varPhi$ with $\CF(\varPhi)= C'$. In this case, $\varPhi$ is called a \Defn{collapsing sequence} from $C$ to $C'$. 
\end{compactitem}
Recall that a \Defn{perfect Morse matching} $\varPhi$ on a regular CW complex $C$ is a Morse matching with $c_i(\varPhi)=\beta_i(C)$ for all $i$, where $\beta_i$ denotes the $i$-th rational Betti number. Recall also that a \Defn{model} for a topological space is a CW complex homotopy equivalent to it, and that a topological space is \Defn{minimal} if it has a model $\Sigma$ consisting of precisely $\beta_i(\Sigma)$ cells of dimension $i$ for all $i$.

\begin{cor}\label{cor:MorseThm}
Let $C$ denote a regular CW complex that admits a perfect Morse matching. Then $C$ is minimal.
\end{cor}

Let $C$ denote a regular CW complex. Let $\sd C$ denote any (simplicial) complex combinatorially equivalent to the order complex of the face poset $\mathcal{P}(C)$ of nonempty faces of $C$. The complex $C$ is a \Defn{(closed) PL $d$-manifold} if the link of every vertex of $\sd C$ has a subdivision that is combinatorially equivalent to some subdivision of the boundary of the $(d+1)$-simplex.

Now, recall that the faces of $\sd C$ correspond to chains of inclusion in $\mathcal{P}(C)$. To any face $\sigma$ of $C$ we associate the union $\sigma^\ast$ of faces of $\sd C$ which correspond to chains in $\mathcal{P}(C)$ with minimal element $\sigma$. Assume now $C$ is a closed PL manifold of dimension $d$, then $\rint \sigma^\ast$, the cell \Defn{dual} to $\sigma$ in $C$, is an open ball of dimension $d-\dim \sigma$. The collection $C^\ast$ of cells $\rint \sigma^\ast,\ \sigma \in \F(C)$, together with the canonical attaching homeomorphisms, forms a regular CW complex, the \Defn{dual block complex} to $C$ (for details, we refer the reader to Proposition 4.7.26(iv) in~\cite{BLSWZ}). Naturally, $\sigma^\ast$ is the \Defn{dual face} to $\sigma$.

If $\varPhi$ is a Morse matching on $C$, then a matching $\varPhi^\ast$, the \Defn{dual} to $\varPhi$, is induced by the map $\sigma\mapsto\sigma^\ast$ as follows: \[\varPhi^\ast:=\{(\sigma,\varSigma)^\ast : (\sigma,\varSigma)\in \varPhi \},\ \text{where}\ (\sigma,\varSigma)^\ast{:=}(\varSigma^\ast,\sigma^\ast).\]

\begin{theorem}[cf.\ Benedetti {\cite[Thm.\ 3.10]{B-DMT4MWB}},\ Forman {\cite[Thm.\ 4.7]{FormanADV}}]\label{thm:dual}
Let $C$ be a regular CW complex that is also a closed PL $d$-manifold, and let $\varPhi$ denote a Morse matching on $C$. Then $\varPhi^\ast$ is a Morse matching on the regular CW complex $C^\ast$, and the map assigning each face of $C$ to its dual in $C^\ast$ restricts to a natural bijection from $\CF(\varPhi)$ to $\CF(\varPhi^\ast)$.
\end{theorem}

\subsection{2-arrangements, combinatorial stratifications, complement complexes} \label{ssc:2-arrangements}

In this section, we introduce $\cc$-arrangements, their combinatorial stratifications and complement complexes, guided by~\cite{BjZie}. All subspace arrangements considered in this exposition are finite. 

A \Defn{$\cc$-arrangement} $\HA$ in $S^d$ (resp.\ in $\R^d$) is a finite collection of distinct ${(d-\cc)}$-dim\-ensional subspaces of $S^d$ (resp.\ of $\R^d$), such that the codimension of any non-empty intersection of its elements is a multiple of~$\cc$.  Any $\cc$-arrangement with $d<\cc$ is the empty arrangement. 
A subspace arrangement $\HA$ is \Defn{essential} if the intersection of all elements of $\HA$ is empty, and \Defn{non-essential} otherwise. By convention, the non-essential arrangements include the empty arrangement. The following definitions apply more generally to the bigger class of \Defn{codim-$\cc$-arrangements}, which are finite collections of subspaces of codimension $\cc$, to allow us to pass between $\R^d$ and $S^d$ without problems, cf.\ Remark~\ref{rem:cod}. 

Recall that an interval in $\mathbb{Z}$ is a set of the form $[a,b]:=\{x\in \mathbb{Z}: a\le x\le b\}.$

\begin{definition}[Extensions and stratifications]
A \Defn{sign extension} $\SH$ of a codim-$\cc$-arrangement $\HA=\{h_i : i\in [1,n]\}$ in $S^d$ is any collection of hyperplanes $\{H_{i}\subset S^d : {i\in [1,n]}\}$ such that for each $i$, we have $ h_i\subset H_{i}$. We say that the subspaces $S^d$, $H_i$ and $h_i$ itself \Defn{extend} $h_i$. A \Defn{hyperplane extension} $\EH$ of $\HA$ in $S^d$ is a sign extension of $\HA$ together with an arbitrary finite collection of hyperplanes in $S^d$.

 Consider now any subset $S$ of elements of $\EH$ and any point $x$ in $S^d$. The \Defn{stratum} associated to $x$ and $S$ is the set of all points that can be connected to $x$ by a curve that lies in all elements of $S$, and intersects no element of $\EH\setminus S$, cf.~\cite[Sec. 2]{BjZie}. The nonempty strata obtained this way form a partition of $S^d$ into convex sets, the \Defn{stratification} of $S^d$ given by~$\EH$. To shorten the notation, we will sometimes say that a stratification $\s$ is induced by a codim-$\cc$-arrangement $\HA$ if it is given by some extension $\EH$ of $\HA$.
\end{definition}

The inclusion map of closures of strata of a stratification gives rise to canonical attaching homeomorphisms of the strata to each other.

\begin{definition}[Fine extensions and combinatorial stratifications]
Let $\HA$ be a {codim-$\cc$-arrangement}. We say that an extension $\EH$ of $\HA$ is \Defn{fine} if it gives rise to a stratification $\s(\EH)$ that, together with the canonical attaching maps, is a regular CW complex; in this case, the stratification $\s(\EH)$ is \Defn{combinatorial}. For instance, if $\HA$ is essential, then any stratification $\s$ of $S^d$ induced by it is combinatorial.
\end{definition}

Let $C$ be a subcomplex of a combinatorial stratification of $S^d$. Let $M$ be an arbitrary subset of $S^d$. The \Defn{restriction} $\RS(C,M)$ of $C$ to $M$ is the maximal subcomplex of $C$ all whose faces are contained in~$M$. If $D$ is a subcomplex of $C$ then $C-D:=\RS(C, C{\setminus}\, \rint D)$ is the \Defn{deletion} of $D$ from $C$. 

Dually, let $\s^\ast$ be the dual block complex of a combinatorial stratification $\s$ of $S^d$, and let $C$ be any subcomplex of $\s^\ast$. Then, $\RS^\ast(C,M)$ is the minimal subcomplex of $C\subset \s^\ast$ containing all those faces of $C$ that are dual to faces of $\s$ intersecting~$M$.

\begin{definition}[Complement complexes for spherical codim-$\cc$-arrangements]
Let $\HA$ be a codim-$\cc$-arrangement in $S^d$, and let $\s$ be a combinatorial stratification of $S^d$ induced by $\HA$. With this, define the \Defn{complement complex} of $\HA$ with respect to $\s$ as the regular CW complex \[\K(\HA,\s):=\RS^\ast(\s^\ast,S^d{\setminus} \HA).\] Equivalently, $\K(\HA,\s)$ is the subcomplex of $\s^\ast$ consisting of the duals of the faces that are \emph{not} contained in $\RS(\s,\HA)$.
\end{definition}

\begin{definition}[Complement complexes for affine codim-$\cc$-arrangements]\label{def:affine}
Let $\HA$ denote a codim-$\cc$-arrangement of affine subspaces in $\R^{d}$, and let $\rho$ denote a radial projection of $\R^{d}$ to an open hemisphere $\OD$ in $S^d$. Extend the image $\rho(\HA):=\{\rho(h): h\in \HA\}$ to the codim-$\cc$-arrangement \[\HA':=\{\SSp(h)\subset S^d: h\in \rho(\HA)\}\] in $S^d$, and
consider any combinatorial stratification $\s$ of $S^d$ induced by $\HA'$. We say that $\RS^\ast(\s^\ast,\OD{\setminus} \HA')$ is a \Defn{complement complex} of $\HA$.
\end{definition}

\begin{rem}[$\cc$-arrangement with respect to an open hemisphere] \label{rem:cod}
Definition~\ref{def:affine} is the reason for defining stratifications in the generality of codim-$\cc$-arrangements; if $\HA$ is an affine $\cc$-arrangement in $\R^d$, then $\HA'$ (compare the preceding definition) \emph{is not} in general a $\cc$-arrangement in $S^d$. However, $\HA'$ \emph{is} a \Defn{$\cc$-arrangement w.r.t.\ $\OD$} in $S^d$, that is, every non-empty intersection of elements of $\HA'$ and $\OD$ has a codimension divisible by $\cc$.
\end{rem}

\begin{lemma}[cf.\ {\cite[Prp.\ 3.1]{BjZie}}]
Let $\HA$ be a codim-$\cc$-arrangement in $\R^d$ (resp.~$S^d$). Then every complement complex of $\HA$ is a model for the complement $\HA^{\comp}=\R^d{\setminus}\HA$ (resp.\ $S^d{\setminus}\HA$).
\end{lemma}

\subsection{Outwardly matched faces of subcomplexes; out-{\em j} collapses} \label{ssc:relcollapse}

Here, we introduce the notion of outward matchings. We will see in the next section that this notion allows us to define a rudimentary Alexander duality for Morse matchings, which is crucial to our proofs.

\begin{definition}[Outwardly matched faces]
Let $C$ be a regular CW complex. Let $D$ be a non-empty subcomplex. Consider a Morse matching $\varPhi$ on $C$. A face $\tau$ of $D$ is \Defn{outwardly matched} with respect to the pair $(C,D)$ if it is matched with a face that does not belong to $D$.
\end{definition}

\begin{definition}[Out-$j$ collapse of a pair, Out-$j$ collapsibility]
Let $C$ be a regular CW complex. Let $D$ be a subcomplex. Suppose that $C$ collapses onto a subcomplex $C'$. The pair $(C,  D)$ \Defn{out-$j$ collapses} to the pair $(C', D \cap C')$, and we write \[(C, D) \searrow_{\textrm{out-}j}  (C', D \cap C' ),\]
if the collapsing sequence that reduces $C$ to $C'$ can be chosen so that every outwardly matched face with respect to the pair $(C,D)$ has dimension $j$.

We say that the pair $(C, D)$ is \Defn{out-$j$ collapsible} if there is a vertex $v$ of $D$ such that \[(C, D) \searrow_{\textrm{ out-}j} (v, v).\] For any integer $j$, the pair $(C, \emptyset)$ is \Defn{out-$j$ collapsible} if $C$ is collapsible. A collapsing sequence demonstrating an out-$j$ collapse is called an \Defn{out-$j$ collapsing sequence.}
\end{definition}

\begin{example} \label{ex:out}
Let $C$ be a collapsible complex. 
\begin{compactenum}[(a)]
\item The pairs $(C, C)$ and $(C, \emptyset)$ are out-$j$ collapsible for any $j$.
\item If $D$ is the $k$-skeleton of $C$, with $0 \le k < \dim C$, then $(C, D)$ is out-$k$ collapsible. 
\item If $D$ is any triangulation of the Dunce hat (cf.~\cite{Zeeman}) that is a subcomplex of $C$, then $(C, D)$ is not out-$j$ collapsible for any $j$, cf.\ Proposition~\ref{prp:relind}.
\end{compactenum}
\end{example}

\begin{prp}\label{prp:relind}
Let $(C,D)$ be an out-$j$ collapsible pair, with $D$ non-empty.  
The number of outwardly matched $j$-faces of $D$ is independent of the out-$j$ collapsing sequence chosen, and equal to 
$(-1)^j \cdot ( \chi(D) - 1 )$, where $\chi$ denotes the Euler characteristic. Moreover, the following are equivalent:
\begin{compactenum}[\rm (1)]
\item $D$ is contractible;
\item There exists a collapsing sequence that has no outwardly matched faces with respect to the pair $(C,D)$;
\item $D$ is collapsible.
\end{compactenum}
\end{prp}

\begin{definition}
With the above notation, if one of the conditions (1), (2), or (3) is satisfied, we say that $(C,D)$ is a \Defn{collapsible pair}. The pair $(C,\emptyset)$ is a \Defn{collapsible pair} if $C$ is collapsible. 
\end{definition}

\begin{proof}[\textbf{Proof of Proposition~\ref{prp:relind}}] Fix an out-$j$ collapsing sequence for $(C,D)$. Let $O$ be the set of outwardly matched faces of $D$. Let $v$ be the vertex onto which $C$ collapses. Let $Q$ be the set of the faces of $D$ that are matched together with another face of $D$. Clearly, the sets $O, \{v\}, Q$ form a partition of the set of all faces of $D$. We set
\[f_i := \# \{ \textrm{$i$-faces of $D$} \}, 
\quad o_i := \# \{ \textrm{$i$-faces in $O$} \}, 
\quad q_i := \# \{ \textrm{$i$-faces in $Q$} \}.\]
Clearly, $f_0 = o_0 + 1 + q_0$, and $f_i = o_i + q_i$ for $i \ge 1$. By definition of out-$j$ collapsibility, $o_i = 0$ unless~$i=j$. In particular, $\sum (-1)^i o_i = (-1)^j o_j$.
Now, faces matched together in a collapsing sequence must have consecutive dimension. It follows (by pairwise canceling) that $\sum (-1)^i q_i = 0$. Hence, 
\[\chi(D) \: = \; \sum (-1)^i f_i \; = \; 1 + \sum (-1)^i o_i + \sum (-1)^i q_i \; = \: 1 + (-1)^j o_j.\]
Hence $o_j = (-1)^j \cdot ( \chi(D) - 1 )$ and the first claim is proven. For the second part: 
\begin{compactitem}[$\circ$]
\item \makebox[4.4em][l]{(1) $\Rightarrow$ (2)}: If $D$ is contractible, $\chi(D) =1$. By the formula above, $o_j = 0$.
\item \makebox[4.4em][l]{(2) $\Rightarrow$ (3)}: By assumption, there is a collapsing sequence for $C$ which removes all faces of $D$ in pairs. Consider the restriction to $D$ of the collapsing sequence for $C$; this yields a collapsing sequence for $D$.
\item \makebox[4.4em][l]{(3) $\Rightarrow$ (1)}: This is implied by the fact if $D\searrow D'$, then $D'$ is a deformation retract of $D$. \qedhere
\end{compactitem}
\end{proof}

We also have the following elementary Lemma.

\begin{lemma}\label{lem:outcoll}
Let $C$ and $D$, $D \subset C$, be regular CW complexes, and let $v$ be any vertex of $C$. Assume that $v\notin D$ or $\Lk(v,D)$ is nonempty, and that $(\Lk(v,C),\Lk(v,D))$ is out-$j$ collapsible. Then $(C,D)$ out-$(j+1)$ collapses to $(C-v,D-v)$. In particular, if in this situation $(\Lk(v,C),\Lk(v,D))$ is a collapsible pair, then $(C,D)$ out-$j$ collapses to $(C-v,D-v)$ for every $j$.

If on the other hand $v\in D$ but $\Lk(v,D)$ is empty, then $(C,D)$ out-$0$ collapses to $(C-v,D-v)$ if and only if $\Lk(v,C)$ is collapsible. \qed
\end{lemma}

\subsection{Complement matchings}\label{ssc:cmpm}

Let $\s$ denote a combinatorial stratification of the sphere $S^d$. As such, $\s$ is necessarily a closed PL manifold. To obtain Morse matchings on the complement complex $\K:=\K(\HA,\s)$, we first study Morse matchings on the stratification $\s$. Via duality, Morse matchings on $\s$ will then give rise to Morse matchings on $\K$. To explain the details of this idea is the purpose of this section.

\begin{definition}[Restrictions of matchings]
Let $\varPhi$ be a Morse matching on a regular CW complex $C$, and let $D$ be a subcomplex of $C$. Let us denote by $\varPhi_D$ the \Defn{restriction} of $\varPhi$ to $D$, that is, the collection of all matching pairs in $\varPhi$ involving two faces of $D$. 
\end{definition}

\begin{example}[Complement matching]\label{ex:cm}
Let $\HA$ be a codim-$\cc$-arrangement in $S^d$, and let $\s$ denote a combinatorial stratification of $S^d$ induced by it. Let $\K:=\K(\HA,\s)$ denote the associated complement complex. Consider the dual Morse matching $\varPhi^\ast$ on $\s^\ast$ of a Morse matching $\varPhi$ defined on $\s$.  The Morse matching $\varPhi^\ast$ has an outwardly matched face $\varSigma^\ast$ (matched with a face $\sigma^\ast$) with respect to the pair $(\s^\ast,\K)$ for every outwardly matched face $\sigma$ (matched with $\varSigma$) of $\varPhi$ with respect to the pair $(\s,\RS(\s,\HA))$. After we remove all such matching pairs from $\varPhi^\ast$, we are left with a Morse matching on $\s^\ast$ that has no outwardly matched faces with respect to the pair~$(\s^\ast,\K)$. If we furthermore remove all matching pairs that only involve faces of $\s^\ast$ not in $\K$, we obtain a Morse matching on the complex $\K$, the \Defn{complement matching}~$\varPhi^{\ast}_{\K}$ induced by $\varPhi$.
\end{example}

\begin{figure}[htbf]
\centering 
 \includegraphics[width=0.60\linewidth]{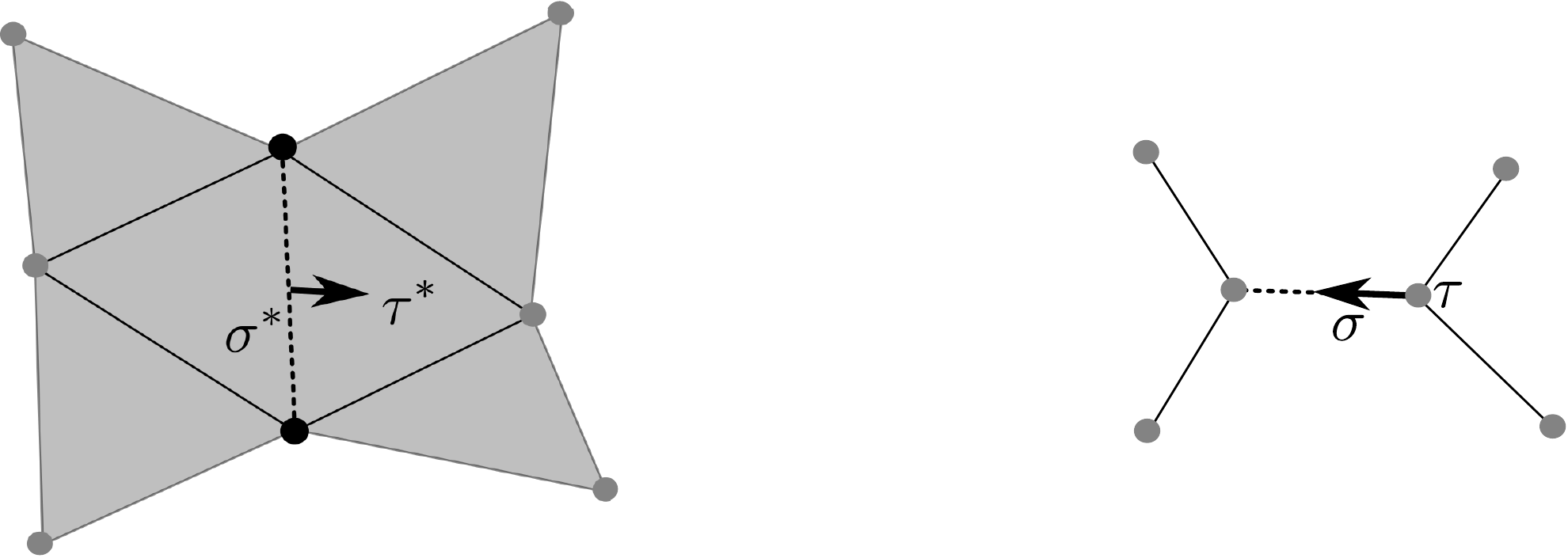} 
 \caption{\small An outward matching $(\varSigma^\ast,\sigma^\ast)$ of the pair $(\s^\ast,\K)$ corresponds to an outward matching $(\sigma,\varSigma)$ of the pair $(\s,\RS(\s,\HA))$.}
\label{fig:outward}
\end{figure}

The complement matching of a Morse matching is again a Morse matching. This allows us to study Morse matchings on the complement complex by studying Morse matchings on the stratification itself. The following theorem can be seen as a very basic Alexander duality for Morse matchings.

\begin{theorem}\label{thm:relcoll}
Let $\s$, $\HA$ and $\K$ be given as in Example~\ref{ex:cm}. Consider a Morse matching $\varPhi$ on $\s$. Then the critical $i$-faces of $\varPhi^{\ast}_{\K}$ are in one-to-one correspondence with the union of
\begin{compactitem}[$\circ$]
\item the critical $(d-i)$-faces of $\varPhi$ that are not faces of $\RS(\s,\HA)$, and
\item the outwardly matched $(d-i-1)$-faces of $\varPhi$ with respect to the pair $(\s,\RS(\s,\HA))$.
\end{compactitem}
\noindent If $M$ is furthermore an open subset of $S^d$ such that all noncritical faces of $\varPhi$ intersect $M$, then the critical $i$-faces of $\varPhi^{\ast}_{\RS^\ast(\K,M)}$ are in bijection with the union of
\begin{compactitem}[$\circ$]
\item the critical $(d-i)$-faces of $\varPhi$ that are not faces of $\RS(\s,\HA)$, and that intersect $M$, and
\item the outwardly matched $(d-i-1)$-faces of $\varPhi$ with respect to the pair $(\s,\RS(\s,\HA))$.
\end{compactitem}
\end{theorem}

\begin{proof} Both bijections are natural: Each critical $i$-face of $\varPhi^{\ast}_{\K}$ corresponds to either a critical $i$-face of $\varPhi^\ast$ in $\K$ or to an outwardly matched $i$-face of $\varPhi^\ast$ with respect to the pair $(\s^\ast, \K)$. Now, 
\begin{compactitem}[$\circ$]
\item critical $i$-faces of $\varPhi^\ast$ in $\K$ are in one-to-one correspondence with the critical $(d-i)$-faces of $\varPhi$ that are not faces of $\RS(\s,\HA)$ (cf.\ Theorem~\ref{thm:dual}), and
\item the outwardly matched $i$-faces of $\varPhi^\ast$ with respect to the pair $(\s^\ast, \K)$, are in one-to-one correspondence with the outwardly matched $(d-i-1)$-faces of $\varPhi$ with respect to the pair  $(\s,\RS(\s,\HA))$ (cf.\ Example~\ref{ex:cm}).
\end{compactitem}
This gives the desired bijection for the critical faces of $\varPhi^{\ast}_{\K}$. The bijection for the critical faces of $\varPhi^{\ast}_{\RS^\ast(\K,M)}$ is obtained analogously.\end{proof} 

\subsection{Discrete Morse theory and duality, II: A strong version of the Morse Theorem}\label{sec:morsethm}

The notions introduced in the past sections allow us to state the following stronger version of Theorem~\ref{thm:MorseThmw}. Recall that a \Defn{$k$-cell}, or \Defn{cell of dimension $k$} is just a $k$-dimensional open ball. We say that \Defn{a cell $B$ of dimension $k$ is attached to a topological space $X$} if $B$ and $X$ are identified along an inclusion of $\partial B$ into~$X$.

\begin{theorem}[Forman {\cite[Thm.\ 3.4]{FormanADV}}]\label{thm:MorseThm}
Let $C$ be a regular CW complex, and let $D$ denote any subcomplex. Let $\varPhi$ denote a Morse matching on $C$ that does not have any outwardly matched faces with respect to the pair $(C,D)$. Then $C$ is up to homotopy equivalence obtained from $D$ by attaching one cell of dimension $k$ for every critical $k$-face of $\varPhi$ not in $D$.
\end{theorem}

\begin{proof}
Theorem 3.4 of Forman in~\cite{FormanADV} treats the case where $D$ contains all but one of the critical faces of $\varPhi$; this is clearly sufficient to prove the statement. For the sake of completeness, we sketch a reasoning using the language of Morse matchings here, inspired by Chari's proof of Theorem~\ref{thm:MorseThmw}~\cite[Thm.~3.1]{Chari}. Assume that $D$ is a strict subcomplex of $C:=C_0$. Set $\varPhi_0:=\varPhi$ and $i:=0$.

\medskip

\noindent {\bf Deformation process} Let $G(C_i)$ denote the Hasse diagram of $\mathcal{P}(C_i)$, i.e.\ let $G(C_i)$ be the graph
\begin{compactitem}[$\circ$]
\item whose vertices are the nonempty faces of $C_i$, and for which
\item two faces $\tau$, $\sigma$ are connected by an edge, directed from $\sigma$ to $\tau$, if and only if $\tau$ is a facet of $\sigma$. 
\end{compactitem}
We manipulate $G(C_i)$ to a directed graph $G_{\varPhi_i} (C_i)$ as follows: 

\begin{quote}
\noindent \emph{For every matching pair $(\sigma, \varSigma)$ of $\varPhi_i$, replace the edge directed from $\varSigma$ to $\sigma$ by an edge directed from $\sigma$ to $\varSigma$.}
\end{quote}
Finally, contract the vertices of $G_{\varPhi_i} (C_i)$ corresponding to $D$ to a single vertex, obtaining the directed graph $G_{\varPhi_i} (C_i)\cdot D$. Let $v_D$ denote the vertex corresponding to $D$ in that graph. Since $\varPhi_i$ contains no closed gradient path, and $(C_i,D)$ has no outwardly matched face with respect to $\varPhi_i$, the directed graph $G_{\varPhi_i} (C_i)\cdot D$ is \Defn{acyclic} (i.e.\ it contains no directed cycle) and $v_D$ is a \Defn{sink} (i.e.\ every edge of the graph that contains $v_D$ points towards it).

Consequently, $G_{\varPhi_i} (C_i)\cdot D$ has a \Defn{source} that is not $v_D$, i.e.\ a vertex such that every edge containing it points away from it. Indeed, to find a source, pick any vertex $x$ of $G_{\varPhi_i} (C_i)\cdot D$ that is not $v_D$ (since $v_D$ might be isolated), and pass to any vertex $y$ connected to $x$ if the edge between them is directed from $y$ to $x$; repeating this procedure will lead us to the desired source since $G_{\varPhi_i} (C_i)\cdot D$ is acyclic. 

The source vertex corresponds to a face $\sigma$ of $C_i$ not in $D$, which, since it is a source, must satisfy one of the following properties:
\begin{compactenum}[(1)]
\item there exists a face $\varSigma$ of $C_i$ such that $(\sigma,\varSigma)$ is a matching pair of $\varPhi_i$, or 
\item $\sigma$ is a critical face of $\varPhi_i$.
\end{compactenum}
In case $\sigma$ satisfies (1), $\sigma$ is a free face of $C_i$, and $C_i$ elementarily collapses to $C_i-\sigma$; in particular, $C_i$ is homotopy equivalent to $C_i-\sigma$. In case $\sigma$ satisfies (2), $\sigma$ is a facet of $C_i$: in particular, $C_i$ is obtained from $C_i-\sigma$ by attaching a cell of dimension $\dim \sigma$. 

Now, set $C_{i+1}:=C_i-\sigma$ and
\[
\varPhi_{i+1}:= \left\{ \begin{array}{ll}\varPhi_i \setminus \{(\sigma,\varSigma)\}&\text{ in case $\sigma$ satisfies (1) and}\\
\varPhi_i                             & \text{ in case $\sigma$ satisfies (2).}
\end{array}
\right.
\]
With this definition, we have that  
\begin{equation}\tag{$\ast$} \label{eq:cr}
c_k(\varPhi_{i+1})= \left\{ \begin{array}{ll}c_k(\varPhi_i) &\text{ if $\sigma$ satisfies (1) or $k\neq \dim \sigma$ and }\\
c_k(\varPhi_i)-1                            & \text{ if $\sigma$ satisfies (2) and $k=\dim \sigma$.}
\end{array}
\right.
\end{equation}
If $C_{i+1}=D$, stop the deformation process. If $C_{i+1}\neq D$, increase $i$ by one and repeat from the start.

\medskip

The homotopical characterization of how to obtain $C_{i+1}=C_i-\sigma$ from $C_{i}$, together with Equation~\eqref{eq:cr}, gives the desired presentation for $C$ from $D$.
\end{proof}

\section{Restricting stratifications to general position hemispheres}

In this section, we study Morse matchings on combinatorial stratifications of $S^d$. More precisely, we study Morse matchings on the restrictions of stratifications to a hemisphere. The main result of this section is Theorem~\ref{thm:hemisphere}, which will turn out to be crucial in order to establish Main Theorem~\ref{MTHM:BZP}. 

If $\HA$ is a subspace arrangement in $S^d$, and $H$ is a subspace of $S^d$, we define \[\HA^H:=\{h\cap H: h \in \HA\}.\] 
We use $\RN^1_p X$ to denote the subset of unit vectors in the tangent space of $X$ at $p$. If $\HA$ is a subspace arrangement in $S^d$, we define the \Defn{link} $\Lk(p,\HA)\subset \RN^1_p S^d$ of $\HA$ at $p$ by \[\Lk(p,\HA):=\{\RN^1_p h: h \in \HA,\ h\cap p\ne \emptyset\}.\] 
Similarly, if $C$ is a subcomplex of a combinatorial stratification of $S^d$, and $v$ is a vertex of $C$, the \Defn{link} $\Lk(v,C)\subset \RN^1_v S^d$ of $C$ at $v$ is the regular CW complex represented by the collection of faces \[\F(\Lk(v,C)):=\{\RN^1_v \sigma: \sigma \in \F(C),\ v\subset \sigma\}.\]
Our goal is to investigate whether, for any given hemisphere $\FD$, the pair $(\RS(\s,\FD), \RS(\s,\FD\cap \HA))$ is out-$j$ collapsible for some suitable integer $j$. With an intuition guided by the Lefschetz hyperplane theorems for complex varieties, one could guess that the right $j$ to consider is the integer 
\[\ii{d}  :=  \left\lfloor \nicefrac{d}{\cc} \right\rfloor  - 1.\]
This will turn out to be correct. Before we start with the main theorem of this section, we anticipate a special case: we consider the case of the empty arrangement.

\begin{lemma}\label{lem:hemisphere}
Let $\EH$ be a fine extension of the empty arrangement in $S^d$, and let $\s:=\s(\EH)$ be the associated combinatorial stratification of $S^d$. Let $\FD$ be a closed hemisphere that is in general position with respect to $\s(\EH)$. Then $\RS(\s,\FD)$ is collapsible.
\end{lemma}

\begin{proof}
The proof is by induction on the dimension, the case $d=0$ clearly being true. Assume now $d\ge 1$. Let $H$ denote any element of $\EH$, and let $\overline{H}_+$, $\overline{H}_-$ denote the closed hemispheres in $S^d$ bounded by $H$. The proof of the induction step is articulated into three simple parts:
\begin{compactenum}[(1)]
\item We prove $\RS(\s, \FD\cap \overline{H}_+)  \searrow  \RS(\s, \FD\cap H).$
\item We prove $\RS(\s, \FD\cap \overline{H}_-)  \searrow  \RS(\s, \FD\cap H).$
\item We show that $\RS(\s, \FD\cap {H})$ is collapsible.
\end{compactenum}

These three steps show that $\RS(\s, \FD)$ is collapsible: The combination of (1) and (2) gives that $\RS(\s, \FD)$ collapses to $\RS(\s, \FD\cap {H}),$ which is collapsible by step (3). We now show point (1); the proof of (2) is analogous and left out, and (3) is true by induction assumption.

Let $\zeta$ denote a central projection of $\intx \FD$ to $\R^{d}$, and let $\nu_+$ denote the interior normal to the halfspace $\zeta(\overline{H}_+ \cap \intx\FD)\subset \R^{d}$. Perturb $\nu_+$ to a vector $\nu$ such that the function $x \ \mapsto \ \langle \zeta(x),  \nu \rangle$ 
\begin{compactenum}[(a)]
\item \Defn{preserves the order given by $\langle  \zeta(\cdot), \nu_+  \rangle$} and
\item \Defn{induces a strict total order} on $\F_0(\RS(\s, \FD\cap \overline{H}_+))$, 
\end{compactenum}
that is, for any two vertices $v$, $w$ of $\RS(\s, \FD\cap \overline{H}_+)$, we have the following:
\begin{compactenum}[(a)]
\item  
If $ \langle  \zeta(v), \nu_+  \rangle  >  \langle \zeta(w),\nu_+ \rangle$, then $\langle  \zeta(v),  \nu  \rangle  > \langle  \zeta(w), \nu  \rangle$;
\item 
If $ \langle  \zeta(v), \nu  \rangle =  \langle \zeta(w),\nu \rangle$, then $v = w$.
\end{compactenum}
The function $x \mapsto \langle  \zeta(x), \nu  \rangle$ orders the $n$ vertices $v_0,v_1,\, \cdots, v_n$ of $\s$ in the interior of $\FD\cap \overline{H}_+$, with the labeling reflecting the order ($v_0$ is the vertex with the highest value under this function).
Let $\Sigma_i$ denote the complex $\RS(\s, \FD \cap \overline{H}_+)-\{v_0,v_1, \, \cdots, v_{i-1}\}$. 
We demonstrate $\RS(\s, \FD\cap \overline{H}_+)  \searrow  \RS(\s, \FD\cap H)$ by showing that, for all $i\in [0,n]$, we have $\Sigma_i  \searrow  \Sigma_i-v_i =\Sigma_{i+1}.$

To see this, notice that $\Lk(v_i,\s(\EH))$ is a combinatorial stratification of the $(d-1)$-sphere $\RN_{v_i}^1 S^d$, given by the fine hyperplane extension $\Lk({v_i}, \EH)$ of the empty arrangement. The complex $\Lk({v_i},\Sigma_i)$ is the restriction of $\Lk({v_i}, \s(\EH))$ to the general position hemisphere~$\RN_{v_i}^1 \FD_{v_i}$, where
\[\FD_{v_i} \; := \; 
\zeta^{-1}( \{  x\in \R^{d-1} : 
\langle  \zeta({v_i}),  \nu  \rangle 
\; \ge \; 
\langle  x,  \nu  \rangle \}), \]
since the vertices $v_0,\, \cdots, v_{i-1}$ were removed already. Thus, by induction assumption, we have that $\Lk({v_i}, \Sigma_i)$ is collapsible; consequently, $\Sigma_i$ collapses to $\Sigma_{i+1}$, as desired.
\end{proof}

\begin{theorem}\label{thm:hemisphere}
Let $\HA$ be a nonempty $\cc$-arrangement in $S^d$, let $\EH$ be a fine hyperplane extension of $\HA$, and let $\s:=\s(\EH)$ denote the combinatorial stratification of $S^d$ given by it. Let $\FD$ be a closed hemisphere that is in general position with respect to $\s$. Then, for any $k$-dimensional subspace $H$ of $\SH\subset\EH$ extending an element of~$\HA$, we have the following:
\begin{compactenum}[\rm (A)]
\item The pair $(\RS(\s,\FD\cap H), \RS(\s,\FD\cap \HA\cap H))$ is out-$\ii{d}$ collapsible. 
\item If $\HA$ is additionally non-essential, then  $(\RS(\s,\FD\cap H), \RS(\s,\FD\cap \HA\cap H))$ is a collapsible pair. 
\end{compactenum}
\end{theorem}

\begin{proof}
To simplify the notation, we set $\RS{[M]}  :=  \RS \left(  \s,  M\right)$ and $\RS'{[M]}  :=  \RS \left( \s,  M \cap \HA  \right)$ for any subset $M$ of~$S^d$.
We proceed by induction on $d$ and $k$. Let $\io_{d,k}$ denote the statement that part (A) is proven for all spheres of dimension $d$ and subspaces $H$ of dimension $k$. Let $\iit_{d,k}$ denote the statement that part (B) holds for arrangements in spheres of dimension $d$ and subspaces $H$ of dimension $k$. Since $H$ extends an element of $\HA$, we always have $d\ge k \ge d-2$.

For the base cases of the induction, it suffices to treat the cases $\io_{2,0}$ and $\io_{3,1}$. In both cases, $H$ is an element of $\HA$, so $\RS{[\FD\cap H]}= \RS'{[\FD\cap H]}$. Since a pair $(C,C)$ is a collapsible pair if and only if $C$ is collapsible, it suffices to prove that $\RS{[\FD\cap H]}$ is collapsible; in case $\io_{2,0}$, the complex $\RS{[\FD\cap H]}$ is a vertex; in case $\io_{3,1}$, the complex $\RS{[\FD\cap H]}$ is a tree; in both cases, the complex is trivially collapsible. Assume now $d\ge 2$ and $k\ge 0$. Our inductive proof proceeds like this:

\smallskip

\begin{compactenum}[\bf I.]
\item We prove that $\io_{d,k}$ implies $\iit_{d,k}$ for all $d$, $k$.
\item We prove that $\io_{k,k}$ implies $\io_{d,k}$ for $k= d-\cc$.
\item We prove that $\iit_{k-1,k-1}$, $\io_{k-1,k-1}$ and $\io_{d,k-1}$ together imply $\io_{d,k}$ for $k > d-\cc$.
\end{compactenum}

\medskip

\noindent \textbf{Part I. $\io_{d,k}$ implies $\iit_{d,k}$  for all $d$, $k$}

\medskip

Let $\sigma$ denote the intersection of all elements of the non-essential arrangement $\HA$. Since $\HA$ is nonempty, so is $\sigma$ and $\RS'[\FD\cap H]$ deformation retracts onto the contractible complex $\RS[\FD\cap \sigma ]$. Thus, by the second part of Proposition~\ref{prp:relind} and inductive assumption $\io_{d,k}$, the pair $(\RS[\FD\cap H], \RS[\FD\cap \HA\cap H])$ is a collapsible pair. 

\medskip

\noindent \textbf{Part II. $\io_{k,k}$ implies $\io_{d,k}$ for $k= d-\cc$}

\medskip

We have to show that $(\RS{[\FD\cap H]},  \RS'{[\FD\cap H]} )$ is out-$\ii{d}$ collapsible. We will see that it is even a collapsible pair. The $(d-2)$-dimensional subspace $H$ is an element of $\HA$, so that we have \[ (\RS{[\FD\cap H]},  \RS'{[\FD\cap H]} ) = (\RS( \s,  \FD\cap H)  ,  \RS( \s , \FD \cap H \cap \HA) )= ( \RS( \s,  \FD\cap H)  ,  \RS( \s,  \FD\cap H )  ).\]
But a pair $(C,C)$ is a collapsible pair if $C$ is collapsible. By definition, if a pair is out-$\ii{d}$ collapsible, the first complex in the pair is collapsible; so, since the pair $(\RS[\FD\cap H], \RS[\FD\cap H\cap \HA])$ is out-$\ii{d}$ collapsible by inductive assumption $\io_{k,k}$, it follows trivially that $\RS[\FD\cap H]$ is collapsible.

\medskip

\noindent \textbf{Part III. $\iit_{k-1,k-1}$, $\io_{k-1,k-1}$ and $\io_{d,k-1}$ together imply $\io_{d,k}$ for $k > d-\cc$}

\medskip

Let $h$ denote the element of $\HA$ extended by $H$. Let $\eta\in\SH\subset\EH$ be a codimension-one subspace of $H$ that extends $h$ as well. Let $\overline{\eta}_+$ and $\overline{\eta}_-$ be the closed hemispheres in $H$ bounded by $\eta$. We prove that the pair $(\RS{[\FD\cap H]},  \RS'{[\FD\cap H]} )$ is out-$\ii{d}$ collapsible.
The proof consists of three steps:
\begin{compactenum}[(1)]
\item We prove \[(\RS{[\FD\cap \overline{\eta}_+]},  \RS'{[\FD\cap \overline{\eta}_+]} )  \searrow_{\textrm{ out-}\ii{d}}  (\RS{[\FD\cap \eta ]},  \RS'{[\FD\cap {\eta}]} ).\]
\item Symmetrically, we have
\[(\RS{[\FD\cap \overline{\eta}_-]},  \RS'{[\FD\cap \overline{\eta}_-]} )  \searrow_{\textrm{ out-}\ii{d}}  (\RS{[\FD\cap \eta ]},  \RS'{[\FD\cap {\eta}]} ).\]
\end{compactenum}
The combination of these two steps proves 
\[(\RS{[\FD\cap H]},  \RS'{[\FD\cap H]} )  \searrow_{\textrm{ out-}\ii{d}}  (\RS{[\FD\cap \eta ]},  \RS'{[\FD\cap {\eta}]} ).\]
\begin{compactenum}[(1)]
\setcounter{enumi}{+2}
\item It then remains to show that $(\RS{[\FD\cap \eta]},  \RS'{[\FD\cap \eta]} )$ is out-$\ii{d}$ collapsible. This, however, is true by inductive assumption $\io_{d,k-1}$.
\end{compactenum}
\noindent Step (2) is completely analogous step (1), so its proof is left out. It remains to prove (1).
To achieve this, we establish a geometry-based strict total order on the vertices of $\RS{[\FD\cap \overline{\eta}_+]}$, and we collapse them away one at the time. 

In details: Let $\zeta$ be a central projection of $\intx\FD\cap H$ to $\R^{k}$, and let $\nu_+$ denote the interior normal to the halfspace $\zeta(\intx \overline{\eta}_+)\subset \R^{d-1}$. Perturb $\nu_+$ to a vector $\nu$ such that the function $x \mapsto \langle  \zeta(x),  \nu \rangle$ preserves the order given by $\langle  \zeta(\cdot),  \nu_+  \rangle$ and induces a strict total order on $\F_0(\RS(\s, \FD\cap \overline{\eta}_+))$ (see also the proof of Lemma~\ref{lem:hemisphere}). 

The function $\langle \zeta(x), \nu  \rangle$ induces a strict total order on the vertices $v_0,v_1,\, \cdots, v_n$ of $\s$ in the relative interior of $\FD\cap \overline{\eta}_+$, starting with the vertex $v_0$ maximizing it and such that the labeling reflects the order. Set $\Sigma_i:=\RS{[\FD\cap \overline{\eta}_+]}-\{v_0,v_1, \, \cdots, v_{i-1}\}.$ We show (1) by demonstrating that, for all $i\in [0,n]$, we have that \[(\Sigma_i, \RS(\Sigma_i, \HA))  \searrow_{\textrm{ out-}\ii{d}}  (\Sigma_{i+1}, \RS(\Sigma_{i+1}, \HA) )=(\Sigma_i-{v_i},\RS(\Sigma_i,\HA)-{v_i}).\]
To see this, notice that $\Lk({v_i},\s(\EH^H))$ is a combinatorial stratification of the $(k-1)$-sphere $\RN_{v_i}^1 H$, given by the hyperplane extension $\Lk({v_i}, \EH^H)$ of the $\cc$-arrangement $\Lk({v_i},\HA^H)$. The complex $\Lk({v_i}, \Sigma_i )$ is the restriction of $\Lk({v_i}, \s(\EH^H))$ to the general position hemisphere $\RN_{v_i}^1 \FD_{v_i}$, where
\[\FD_{v_i} \; := \; 
\zeta^{-1} (\{  x\in \R^{k-1} : 
\langle  \zeta({v_i}), \nu  \rangle 
\; \ge \; 
\langle  x,  \nu \rangle \}), \]
since ${v_i}$ maximizes $\langle  \zeta(x),  \nu  \rangle$ among the vertices $v_i, v_{i+1}, v_{i+2},\, \cdots$ and the vertices of $\RS[\FD\cap \eta]$. 

At this point, we want to apply the induction assumptions $\io_{k-1,k-1}$ and $\iit_{k-1,k-1}$ to $\Lk({v_i}, \Sigma_i )$. There are two cases to consider.
\begin{compactitem}[$\circ$]
\item If $k=d=1$ mod $\cc$, then $\Lk({v_i},\HA^H)$ is a non-essential $\cc$-arrangement in $\RN_{v_i}^1 H$, and it is nonempty if and only if $v_i$ is in $\RS(\s,\HA)$. Thus, by inductive assumption~$\iit_{k-1,k-1}$ and Lemma~\ref{lem:hemisphere}, we have that the pair 
$\big(\Lk({v_i}, \Sigma_i),\Lk({v_i}, \RS(\Sigma_i,\HA))\big )$
is a collapsible pair. Lemma \ref{lem:outcoll} now proves
\[(\Sigma_i,\RS(\Sigma_i,\HA)) \searrow_{\textrm{ out-}\ii{d}} (\Sigma_i-{v_i},\RS(\Sigma_i,\HA)-{v_i}).\]

\item If $k = 0$ mod $\cc$ or $k=d-1=1$ mod $\cc$, then $\Lk({v_i},\HA^H)$ is a $\cc$-arrangement in $\RN_{v_i}^1 H$. By inductive assumption~$\io_{k-1,k-1}$ and Lemma~\ref{lem:hemisphere}, the pair $\big(\Lk({v_i}, \Sigma_i),\Lk({v_i}, \RS(\Sigma_i,\HA)) \big)$
is an out-$\ii{k-1}$ collapsible~pair. Moreover, if $k
>2$, then $v_i\in\RS(\s,\HA)$ implies that $\Lk({v_i},\HA^H)$ is nonempty. 
Finally, we have $\ii{k}=\ii{d}-1$ by assumption on $d$ and $k$. Using Lemma \ref{lem:outcoll}, we consequently obtain that
\[(\Sigma_i,\RS(\Sigma_i,\HA)) \searrow_{\textrm{ out-}\ii{d}} (\Sigma_i-{v_i},\RS(\Sigma_i,\HA)-{v_i}). \qedhere\]
\end{compactitem}
\end{proof}

\section{Proof of Theorem~\ref{MTHM:LEFT}}
We have now almost all the tools to prove our version of the Lefschetz Hyperplane Theorem (Theorem~\ref{thm:lef}); it only remains for us to establish the following lemma:

\begin{lemma}\label{lem:lef}
Let $\FD$, $\FD'$ denote a pair of closed hemispheres in $S^d$. Let $\HA$ denote a $\cc$-arrangement w.r.t.\ the complement $\OD$ of~$\FD$, let $\EH$ be a fine extension of $\HA$, and let $\s:=\s(\EH)$ denote the combinatorial stratification of $S^d$ given by it. If $\FD'$ is in general position with respect to $\s(\EH\cup \partial \FD)$, then 
\[(\RS(\s,\FD\cup \FD'),\RS(\s,(\FD\cup \FD')\cap \HA)) \searrow_{\textrm{ out-}\ii{d}} (\RS(\s,\FD),\RS(\s,\FD\cap \HA)).\]
\end{lemma}

\begin{proof}
Let $\zeta$ denote a central projection of $\OD$ to $\R^d$, and let $\nu_+$ denote the outer normal to the halfspace $\zeta( \OD \cap \FD')\subset \R^{d}$. Perturb $\nu_+$ to a vector $\nu$ such that the function $x \mapsto \langle  \zeta(x),  \nu  \rangle$ preserves the order given by $\langle  \zeta(\cdot),  \nu_+  \rangle$ and induces a strict total order on $\F_0(\RS(\s, \OD\cap \FD'))$ (see also the proof of Lemma~\ref{lem:hemisphere}).

The function $x \mapsto  \langle \zeta(x), \nu \rangle$ gives an order on the vertices $v_0,v_1,\, \cdots, v_n$ of $\RS(\s, \OD\cap \FD')$, starting with the vertex $v_0$ with the {highest} value under this function and such that the vertices are labeled to reflect their order. Set \[\Sigma_i:=\RS(\s, \OD \cap \FD')-\{v_0,v_1, \, \cdots, v_{i-1}\}.\]
In order to prove 
\[(\RS(\s,\FD\cup \FD'),\RS(\s,(\FD\cup \FD')\cap \HA)) \searrow_{\textrm{ out-}\ii{d}} (\RS(\s,\FD),\RS(\s,\FD\cap \HA)),\]
it suffices to prove that, for all $i\in [0,n]$, we have
\[(\Sigma_i, \RS(\Sigma_i,\HA))  \searrow_{\textrm{ out-}\ii{d}}  (\Sigma_i-v_i, \RS(\Sigma_i,\HA)-{v_i} ) =(\Sigma_{i+1}, \RS(\Sigma_{i+1},\HA)).\]
Thus, let $v_i$ denote any vertex of $\s$ in $ \OD \cap \FD'$.

The complex $\Lk(v_i,\s)$ is a combinatorial stratification of the $(d-1)$-sphere $\RN_{v_i}^1 S^d$ given by the fine extension $\Lk(v_i,\EH)$ of the $\cc$-arrangement $\Lk(v_i,\HA)$, and the complex $\Lk(v_i, \Sigma_i)$ is the restriction of $\Lk(v_i, \s)$ to the general position hemisphere $\RN_v^1 \FD_{v_i}$, where
\[\FD_{v_i} \; := \; \zeta^{-1} 
(\{  x\in \R^{k-1} : \langle  \zeta({v_i}),  \nu \rangle \; \ge \; 
\langle  x,  \nu  \rangle \}). \] 
Thus, as in Part III of Theorem~\ref{thm:hemisphere}, there are two cases:
\begin{compactitem}[$\circ$]
\item If $d=1$ mod $\cc$, then $\Lk(v_i,\HA)$ is a non-essential $\cc$-arrangement in $\RN_{v_i}^1 S^d$, and it is nonempty if and only if $v_i$ is in $\RS(\s,\HA)$. Thus, by Theorem~\ref{thm:hemisphere}(B) and Lemma~\ref{lem:hemisphere}, the pair 
$\big(\Lk({v_i},\Sigma_i),\Lk({v_i},\RS(\Sigma_i,\HA))\big)$
is a collapsible pair. Consequently, Lemma \ref{lem:outcoll} proves that the pair $(\Sigma_i, \RS(\Sigma_i,\HA))$ out-$\ii{d}$ collapses to the pair 
\[(\Sigma_i-v_i, \RS(\Sigma_i,\HA) -{v_i}) =(\Sigma_{i+1}, \RS(\Sigma_{i+1},\HA)).\]
\item If $d = 0$ mod $\cc$, then $\Lk({v_i},\HA)$ is a $\cc$-arrangement in $\RN_{v_i}^1 S^d$. By Theorem~\ref{thm:hemisphere}(A) and Lemma~\ref{lem:hemisphere}, the pair 
$ \big(\Lk({v_i},\Sigma_i),\Lk({v_i},\RS(\Sigma_i,\HA))\big)$
is an out-$\ii{d-1}$ collapsible pair. Moreover, if $d> 2$, then $v_i\in\RS(\s,\HA)$ implies that $\Lk({v_i},\HA^H)$ is nonempty. Since $\ii{d-1}=\ii{d}-1$ and by Lemma \ref{lem:outcoll}, we obtain \[(\Sigma_i,\RS(\Sigma_i,\HA)) \searrow_{\textrm{ out-}\ii{d}} (\Sigma_i-{v_i},\RS(\Sigma_i,\HA)-{v_i}). \qedhere\]
\end{compactitem}
\end{proof}

\begin{cor}\label{cor:lef}
Let $\FD$ denote a closed hemisphere in $S^d$, let $\OD$ denote its open complement. Let $\HA$ denote a $\cc$-arrangement w.r.t.\ $\OD$, and let $\s$ denote a combinatorial stratification of $S^d$ induced by~$\HA$. If $H$ is a hyperplane in $S^d$ that is in general position with respect to $\s(\EH\cup \partial \FD)$, then
\[\big(\RS(\s, S^d{\setminus} (\OD \cap H)),\RS(\s,  \HA \cap S^d{\setminus} (\OD \cap H))\big)  \searrow_{\textrm{out-}\ii{d}} (\RS(\s,\FD), \RS(\s,\FD\cap \HA)).\]
\end{cor}

\begin{theorem}\label{thm:lef}
Consider any affine $\cc$-arrangement $\HA$ in $\R^d$, and any hyperplane $H$ in $\R^d$ in general position with respect to $\HA$. Then the complement $\HA^{\comp}$ of $\HA$ is homotopy equivalent to $H\cap\HA^{\comp}$ with $e$-cells attached to it, where $e = \lceil\nicefrac{d}{\cc}\rceil =d-\lfloor\nicefrac{d}{\cc}\rfloor$.
\end{theorem}
\begin{proof}
Define $\rho$, $\OD$ and $\HA'$ in $S^d$ as in Definition~\ref{def:affine}, and define the hyperplane $H_\rho:=\SSp(\rho(H))\subset S^d$ induced by $H$ in $\R^d$. Let $\FD:=\OD^{\comp}$ denote the closed hemisphere complementary to $\OD$. Let $\s$ be a generic combinatorial stratification of $S^d$ induced by $\HA'$. By Corollary~\ref{cor:lef}, \[\big(\RS(\s, S^d{\setminus} (H_\rho \cap \OD)),\RS(\s, \HA'\cap S^d{\setminus} (H_\rho \cap \OD)\big)  \searrow_{\textrm{out-}\ii{d}} (\RS(\s,\FD), \RS(\s,\FD\cap \HA')).\]
The associated out-$\ii{d}$ collapsing sequence gives a Morse matching $\varPhi$ on $\s$ with the following properties:
\begin{compactitem}[$\circ$]
\item the critical faces of $\varPhi$ are the faces of $\RS(\s,\FD)$ and the faces of $\s$ that intersect $H_\rho \cap \OD$;
\item the outwardly matched faces of the pair $(\s, \RS(\s, \HA'))$ are all of dimension $\ii{d}$.
\end{compactitem}
By Theorem~\ref{thm:relcoll} the restriction $\varPhi^{\ast}_{\RS^\ast(\K(\HA',\s), \OD)}$ of the complement matching induced by $\varPhi$ to the complex $\RS^\ast(\K(\HA',\s), \OD)$ has the following critical faces:
\begin{compactitem}[$\circ$]
\item the faces of ${\RS^\ast(\K(\HA',\s), \OD\cap H_\rho)}$, whose duals are critical faces of $\varPhi$, and
\item the critical faces of dimension $e=\lceil\nicefrac{d}{\cc}\rceil=d-\ii{d}-1$, which correspond to the outwardly matched faces of $\varPhi$.
\end{compactitem}
Furthermore, since the faces of $\RS^\ast(\K(\HA',\s), \OD\cap H_\rho)$ are critical in $\varPhi^{\ast}_{\RS^\ast(\K(\HA',\s), \OD)}$, this Morse matching has no outwardly matched faces with respect to the pair \[\big(\RS^\ast(\K(\HA',\s), \OD),\RS^\ast(\K(\HA',\s), \OD\cap H_\rho)\big)\]
Thus, by Theorem~\ref{thm:MorseThm}, we have that \[\RS^\ast(\K(\HA',\s), \OD)\simeq \R^d{\setminus} \HA\] is, up to homotopy equivalence, obtained from \[\RS^\ast(\K(\HA',\s), \OD\cap H_\rho) \simeq H{\setminus} \HA\] by attaching $e$-dimensional cells, as desired.
\end{proof}

\section{Proof of Theorem~\ref{MTHM:BZP}}

We start with an easy consequence of the formula of Goresky--MacPherson. For completeness, we provide the (straightforward) proof in the~appendix.

\begin{lemma}\label{LEM:LHTCA}
Let $\HA$ denote a subspace arrangement in $S^d$, let $\OD$ denote an open hemisphere of $S^d$, and let $H$ be a hyperplane in $S^d$. Then we have the following:
\begin{compactenum}[\rm (I)]
\item If $\OD$ is in general position with respect to $\HA$, then for all $i$, $\beta_i(S^d{\setminus} \HA)\ge\beta_i(\OD{\setminus} \HA).$
\item If $H$ is in general position with respect to $\HA$ and $\OD$, then for all $i$, $\beta_i(\OD{\setminus} \HA)\ge\beta_i((\OD\cap H){\setminus} \HA).$
\end{compactenum}
\end{lemma}
Besides this lemma, we need the following elementary concept. For any convex set $\sigma$ and any hyperplane $H$ in $S^d$, let us denote by $\sigma^H$ the intersection of $\sigma$ with $H$. If $C$ is any collection of polyhedra in~$S^d$, we define $C^H:=\bigcup_{\sigma\in C} \sigma^H$.

\begin{example}[Lifting a Morse matching]
Let $\varSigma$ be a polyhedron in $S^d$. Let $H$ denote a general position hyperplane in $S^d$. Then $\varSigma^H$ is a polyhedron in $H$, and if $\sigma^H$ is a facet of $\varSigma^H$, then there exists a unique facet $\sigma$ of $\varSigma$ with the property that $\sigma^H:=\sigma\cap H$.

Consider now a subcomplex $C$ of a combinatorial stratification of $S^d$, and a Morse matching $\varphi$ on the complex~$C^H$. Then we can match $\sigma$ with $\varSigma$ for every matching pair $(\sigma^H, \varSigma^H)$ in the matching $\varphi$ of $C^H$. This gives rise to a Morse matching $\varPhi$ on $C$ from a Morse matching $\varphi$ on $C^H$, the \Defn{lift} of $\varphi$ to $C$.

\begin{figure}[htbf]
\centering 
 \includegraphics[width=0.64\linewidth]{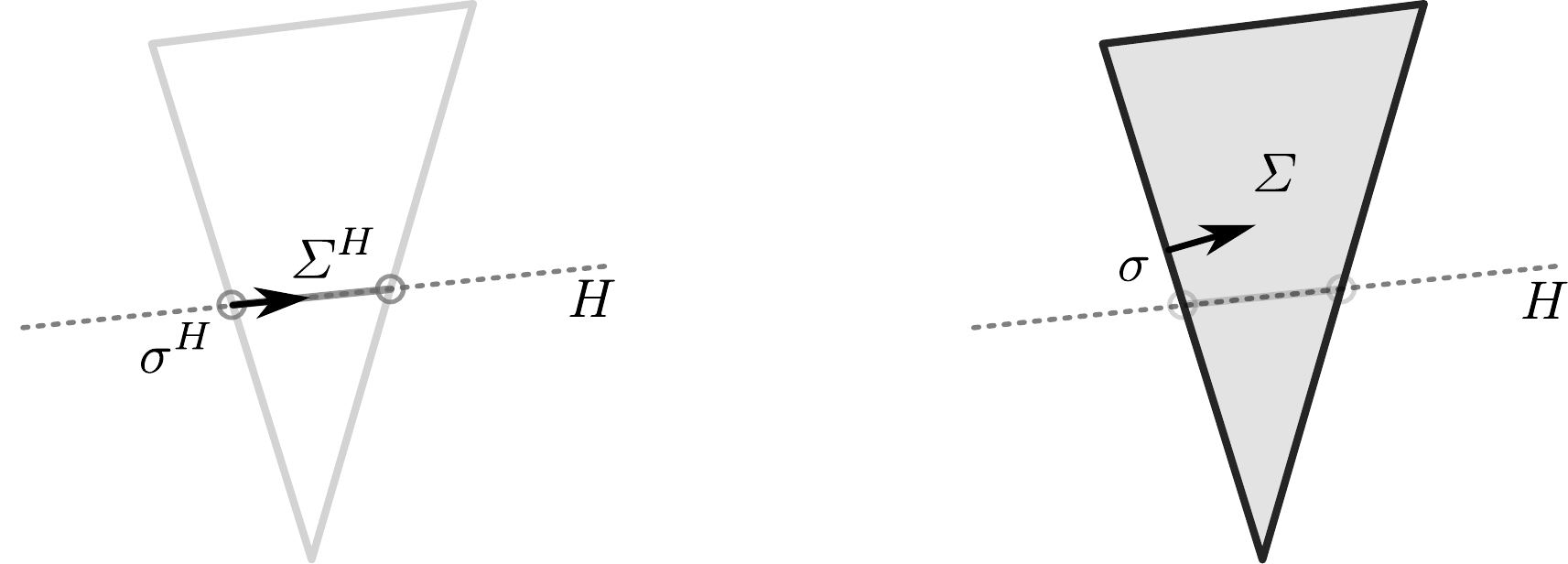} 
 \caption{\small The lift of a matching.}
\label{fig:lift}
\end{figure}
\end{example}

\begin{lemma}\label{lem:1} Let $\FD$ denote a closed hemisphere of $S^d$, and let $\OD:=\FD^{\comp}$ be its complement. Let $\HA$ be a $\cc$-arrangement w.r.t.\ $\OD$, and let $\s$ be a combinatorial stratification of $S^d$ induced by $\HA$. Finally, let $\K:=\K(\HA,\s)$ denote the associated complement complex. 

Then, there exists a Morse matching $\varPsi$ on $\s$ whose critical faces are the subcomplex $\RS(\s, \FD)$ and some additional facet of $\s$ such that the restriction $\varPsi^{\ast}_{\RS^\ast(\K,\OD)}$ of the complement matching to ${\RS^\ast(\K,\OD)}$ is perfect.
\end{lemma}

\begin{proof}
We abbreviate $\RS^\ast[M]:=\RS^\ast(\K,M)$ for any set $M$ in $S^d$. We prove the lemma by induction on the dimension; The case of $d=0$ is clearly true, since in this case $\RS(\s, \OD)$ is just a vertex. Assume now $d\ge 1$. Let $H$ denote a generic hyperplane in $S^d$. By induction on the dimension, there exists a Morse matching $\varphi$ on $\s^H$ such that the restriction $\varphi^{\ast}_{\RS^\ast(\K(\HA^H,\s^H),\OD\cap H)}$ of the complement matching induced by $\varphi$ to $\RS^\ast(\K(\HA^H,\s^H),\OD\cap H)$ is perfect on the latter. We lift $\varphi$ to a Morse matching $\varPhi$ of the faces of $\s$ intersecting~$H$.
Now, by Corollary~\ref{cor:lef}, 
\begin{equation}\tag{$\dagger$} \label{eq:co}
\big(\RS(\s, S^d{\setminus} (\OD \cap H)),\RS(\s,  \HA \cap S^d{\setminus} (\OD \cap H))\big)  \searrow_{\textrm{out-}\ii{d}} (\RS(\s,\FD), \RS(\s,\FD\cap \HA)).
\end{equation}
Denote the associated out-$\ii{d}$ collapsing sequence by $X$. Define the Morse matching $\varPsi$ as the union of the Morse matchings $\varPhi$ and $X$. We claim that $\varPsi$ gives the desired Morse matching on $\s$.

The Morse matching $\varPsi$ has the desired critical faces by construction, so it remains to show that $\varPsi^{\ast}_{\RS^\ast[\OD]}$ is perfect. By construction, $\varPsi^{\ast}_{\RS^\ast[\OD]}$ contains no outwardly matched faces with respect to the pair $(\RS^\ast[\OD],\RS^\ast[\OD\cap H])$. Furthermore, Observation~\eqref{eq:co} and Theorem~\ref{thm:relcoll} show that every critical face of $\varPsi^{\ast}_{\RS^\ast[\OD]}$ that is not in $\RS^\ast[\OD\cap H]$ is of dimension $e:=\lceil\nicefrac{d}{\cc}\rceil=d-\ii{d}-1$. Theorem~\ref{thm:MorseThm} now shows that the complex $\RS^\ast[\OD]$ is obtained, up to homotopy equivalence, from the complex $\RS^\ast[\OD\cap H]$ by attaching cells of dimension $e$. The attachment of each of these cells contributes to the homology of $\RS^\ast[\OD]$ either by deleting a generator of the homology in dimension $e-1$, or by adding a generator for the homology in dimension $e$. But if any of the $e$-cells deletes a generator in homology, then 
\[\beta_{e-1}(\OD{\setminus} \HA)=\beta_{e-1}(\RS^\ast[\OD])<\beta_{e-1}(\RS^\ast[\OD\cap H])=\beta_{e-1}((\OD\cap H){\setminus} \HA),\]
in contradiction with Lemma~\ref{LEM:LHTCA}(II). Thus, every $e$-cell attached must add a generator in homology. Since \[c_i(\varPhi^{\ast}_{\RS^\ast[\OD\cap H]})=c_i (\varphi^{\ast}_{\RS^\ast(\K(\HA^H,\s^H),\OD\cap H)}) =\beta_i((\OD\cap H){\setminus} \HA)\] for all $i$, we consequently have
$c_i(\varPsi^{\ast}_{\RS^\ast[\OD]})=\beta_i\big(\OD{\setminus} \HA\big)$
for all $i$, as desired.
\end{proof}

\begin{theorem} \label{thm:BZperfect}
Any complement complex of any $\cc$-arrangement $\HA$ in $S^d$ or~$\R^d$ admits a perfect Morse matching.
\end{theorem}

\begin{proof}
We distinguish two cases: (1) the case of affine $\cc$-arrangements in $\R^d$, and (2) the case of $\cc$-arrangements in $S^d$.

\begin{compactenum}[(1)]
\item If $\HA$ is any affine $\cc$-arrangement in $\R^{d}$, consider a radial projection $\rho$ of $\R^{d}$ to an open hemisphere $\OD$ in $S^d$, and the arrangement $\HA'$ induced by $\HA$, as defined in Definition~\ref{def:affine}. Any complement complex of the affine $2$-arrangement $\HA$ is of the form $\RS^\ast(\K(\HA',\s),\OD)$, where $\s$ is some combinatorial stratification induced by $\HA'$. The Morse matching $\varPsi^{\ast}_{\RS^\ast(\K(\HA',\s),\OD)}$ constructed in Lemma~\ref{lem:1} provides a perfect Morse function on it.

\item Suppose instead $\HA$ is any $\cc$-arrangement in $S^d$. Any complement complex of $\HA$ is of the form $\K(\HA,\s)$, where $\s$ is some combinatorial stratification induced by $\HA$. Let $\FD$ denote a generic closed hemisphere in $S^d$. By Lemma~\ref{lem:1}, there is a Morse matching $\varPsi$ on $\s$ whose critical faces are the subcomplex $\RS(\s, \FD)$ and some additional facet of $\s$ such that the restriction $\varPsi^{\ast}_{\RS^\ast(\K(\HA,\s),\OD)},\ \OD:=\FD^{\comp},$ of the complement matching to $\RS^\ast(\K(\HA,\s),\OD)$ is perfect. 

By Theorem~\ref{thm:hemisphere}(A), $\RS^\ast(\s,\FD)$ is out-$\ii{d}$ collapsible; by considering the union of the associated out-$\ii{d}$ collapsing sequence with the matching $\varPsi$, we obtain a Morse matching $\varOmega$ on $\K(\HA,\s)$ such that only outwardly matched faces of dimension $\ii{d}$ with respect to the pair $(\s,\RS(\s,\HA))$ are added when passing from $\varPsi$ to $\varOmega$. By Theorems~\ref{thm:relcoll} and~\ref{thm:MorseThm}, $\K(\HA,\s)$ is obtained from $\RS^\ast(\K(\HA,\s),\OD)$, up to homotopy equivalence, by attaching $e$-dimensional cells, where $e:=\lceil\nicefrac{d}{\cc}\rceil=d-\ii{d}-1$. Each of these cells can either add a generator for homology in dimension $e$, or delete a generator for the homology in dimension $e-1$. But if one of the cells deletes a generator, then \[\beta_{e-1}(\RS^\ast(\K(\HA,\s),\OD))<\beta_{e-1}(\K(\HA,\s)),\] which contradicts Lemma~\ref{LEM:LHTCA}(I). Thus, every cell attached adds a generator in homology, and in particular, since 
\[c_i(\varPsi^{\ast}_{\RS^\ast(\K(\HA,\s),\OD\cap H)})=\beta_i\big(\OD{\setminus} \HA\big)\]
for all $i$, we have
$c_i(\varOmega^{\ast})=\beta_i(S^d{\setminus} \HA)$
for all $i$. \qedhere
\end{compactenum}
\end{proof}

\newcommand\Ho{\mathrm{H}}
\appendix
\section{Appendix}
\subsection{Proof of Lemma~\ref{LEM:LHTCA}}
Let us recall the Goresky--MacPherson formula. For an element $p$ of a poset $\mathcal{P}$, we denote by $\mathcal{P}_{<p}$ the poset of elements of $\mathcal{P}$ that precede $p$ with respect to the order of the poset. Let $\Delta(\mathcal{P})$ denote the order complex of a poset $\mathcal{P}$.  Let $\widetilde{\Ho}_\ast$ resp.\ $\widetilde{\Ho}^\ast$ denote reduced homology resp.\ cohomology, and let $\widetilde{\beta}_\ast$ denote the reduced Betti number.

\begin{theorem}[Goresky--MacPherson {\cite[Sec.\ III, Thm.\ A]{GM-SMT}}, {\cite[Cor.\ 2.3]{ZieZiv}}]\label{thm:gmf}
Let $\HA$ denote an arrangement of affine subspaces in $\R^d$. Let $\mathcal{P}(\HA)$ denote the poset of nonempty intersections of elements of $\HA$, ordered by reverse inclusion. Then, for all $i$,
\[\widetilde{\Ho}^i(\R^d{\setminus}{\HA}; \mathbb{Z} )\cong\bigoplus_{p\in \mathcal{P}}\widetilde{\Ho}_{d-2-i-\dim p}\big(\Delta(\mathcal{P}_{<p}(\HA)); \mathbb{Z} \big).\]
\end{theorem}

Apart from the Goresky--MacPherson formula, we will make use of the following observation: We say that a subposet $\mathcal{Q}$ of a poset $\mathcal{P}$ is a \Defn{truncation} of $\mathcal{P}$ if every element of $\mathcal{P}$ not in $\mathcal{Q}$ is a maximal element of~$\mathcal{P}$. In this case, for every element $q$ of $\mathcal{Q}$, we have that $\mathcal{Q}_{<q}=\mathcal{P}_{<q}$, and in particular, $\Delta(\mathcal{Q}_{<q})=\Delta(\mathcal{P}_{<q})$.

\smallskip

We turn to the proof of of Lemma~\ref{LEM:LHTCA}.

\begin{proof}[\textbf{Proof of Lemma~\ref{LEM:LHTCA}}]

The arrangement $\HA$ gives rise to a subspace arrangement $\HA_{\Sp}$ in $\R^{d+1}$ by considering, for every element $H$ of $\HA$, the subspace $\Sp(H)$ in $\R^{d+1}$. If $x$ is the midpoint of the open hemisphere $\OD$ in $S^d\subset \R^{d+1}$, and $H_{\TT}$ denotes the hyperplane in $\R^{d+1}$ tangent to $S^d$ in $x$, then $\HA_{\Sp}^{H_{\TT}}$, defined as the collection of intersections of elements of $\HA_{\Sp}$ with $H_{\TT}$, is a subspace arrangement in $H_{\TT}$. 

Furthermore, if $\zeta$ is a central projection of $\OD$ to $\R^d$, then for every element $h$ of $\HA$, $\zeta(h\cap \OD)$ is an affine subspace in $\R^d$. The collection of subspaces $\zeta(h\cap \OD),\ h\in \HA,$ gives an affine subspace arrangement $\HA_\zeta$ in~$\R^d$. Now, we have that
\begin{compactenum}[(I)]
\item $\R^{d+1}{\setminus} \HA_{\Sp}$ is homotopy equivalent to  $S^{d}{\setminus} \HA$, and $H_{\TT}{\setminus}\HA_{\Sp}^{H_{\TT}}$ is homeomorphic to $\OD{\setminus} {\HA}$; and
\item $\R^{d}{\setminus} {\HA}_{\zeta}$ is homeomorphic to $\OD{\setminus} {\HA}$, and $\zeta(\OD\cap H) {\setminus} {\HA}_{\zeta}$ is homeomorphic to $(\OD\cap H){\setminus} {\HA}$.
\end{compactenum}
Thus, both statements (I), (II) of the lemma are special cases of the following claim for affine subspace arrangements in $\R^d$.

\smallskip 

\noindent \emph{Let $\HA$ denote any affine subspace arrangement in $\R^d$, and let $H$ denote any hyperplane in $\R^d$ in general position with respect to $\HA$. Then, for all $i$ \[{\beta}_i(\R^d{\setminus} \HA)\ge{\beta}_i(H{\setminus} \HA).\]}
\vskip -4mm
\noindent Recall that we use, for an element $p$ of $\mathcal{P}(\HA)$ intersecting the hyperplane $H$, the notation $p^{H}:=p\cap H$ to denote the corresponding subspace of $H$. Recall also that $\HA^{H}$ denotes the affine arrangement obtained as the union of $h^{H}$, $h\in \HA$. Now, $\mathcal{P}(\HA^{H})$ is isomorphic to a truncation of $\mathcal{P}(\HA)$, so for every element $p$ of $\mathcal{P}(\HA)$ intersecting $H$, we have
\[\Delta(\mathcal{P}_{<{p^{H}}}(\HA^{H}))\cong\Delta(\mathcal{P}_{<p}(\HA)).\]
By this observation, and using the Goresky--MacPherson formula (Theorem~\ref{thm:gmf}), we obtain
\[\widetilde{\beta}_i(H{\setminus}{\HA})=\sum_{p\in \mathcal{P}(\HA^{H})}\widetilde{\beta}_{d-2-i-\dim p}\big(\Delta(\mathcal{P}_{<p}(\HA^{H}))\big)=\sum_{\substack {p\in \mathcal{P}(\HA) \\ p\cap H\ne \emptyset}}\widetilde{\beta}_{d-2-i-\dim p}\big(\Delta(\mathcal{P}_{<p}\HA))\big).\] 
Using Theorem~\ref{thm:gmf} once more, we furthermore see that
\[\widetilde{\beta}_i(\R^d{\setminus}{\HA})=\sum_{p\in \mathcal{P}(\HA)}\widetilde{\beta}_{d-2-i-\dim p}\big(\Delta(\mathcal{P}_{<p}(\HA))\big)\ge \sum_{\substack {p\in \mathcal{P}(\HA) \\ p\cap H\ne \emptyset}}\widetilde{\beta}_{d-2-i-\dim p}\big(\Delta(\mathcal{P}_{<p}(\HA))\big).\]
Consequently, we get \[\widetilde{\beta}_i(\R^d{\setminus}{\HA})\ge\widetilde{\beta}_i(H{\setminus}{\HA}). \qedhere\]
\end{proof}

\subsection{Minimality of {\em c}-arrangements}\label{sec:car}
Recall that a \Defn{$c$-arrangement} $\HA$ in $S^d$ (resp.\ in $\R^d$) is a finite collection $(h_i)_{i\in [1,n]}$ of distinct $(d-c)$-dimensional subspaces of $S^d$ (resp.\ of $\R^d$), such that the codimension of any non-empty intersection of its elements is divisible by $c$. 
\begin{quote}
\emph{Is the complement of any $c$-arrangement $\HA$ a minimal space?}
\end{quote}
For $c=2$, the answer is positive, as we saw in Corollary~\ref{mcor:m}. For $c  \neq 2$, the answer is also positive. In fact, the complement $\HA^{\comp}$ of any $c$-arrangement $\HA$ in $S^d$ or $\R^d$, $c\neq 2$, has the following properties:
\begin{compactitem}[$\circ$]
\item $\HA^{\comp}$ has the homotopy type of a CW complex. (This holds for arbitrary subspace arrangements.)
\item $\HA^{\comp}$ has no torsion in homology. (This holds for arbitrary $c$-arrangements, cf.~\cite[Sec.\ III Thm.\ B]{GM-SMT}.)
\item Every connected component of $\HA^{\comp}$ is simply connected. (This is easy to prove, but holds only if $c \neq 2$.)
\end{compactitem} 
Now, \emph{any topological space satisfying these three properties is a minimal space}. This is proven in Hatcher~\cite[Prp.\ 4C.1]{Hatcher}; see also Anick~\cite[Lem.\ 2]{Anick} and Papadima--Suciu~\cite[Rem.\ 2.14]{PapadimaSuciu}. Consequently, Corollary~\ref{mcor:m} can be re-stated as follows:

\begin{cor}
The complement of any $c$-arrangement is a minimal space.
\end{cor}

In the rest of this section, we provide an analogue of Theorem~\ref{MTHM:BZP} for $c$-arrangements. First, some definitions. A \Defn{sign extension} $\SH$ of a $c$-arrangement $\HA=(h_i)_{i\in [1,n]}$ in $S^d$ is any collection of hyperplanes \[(H_{ij})\subset S^d,\ {i\in [1,n],\ j\in[1,c-1]},\] such that $\cap_{j\in [1,c-1]} H_{ij}=h_i$ for each $i$. A \Defn{hyperplane extension} $\EH$ of $\HA$ in $S^d$ is a sign extension of $\HA$ together with an arbitrary (but finite) collection of hyperplanes in $S^d$. The subdivision of $S^d$ into convex sets induced by $\EH$ is the \Defn{stratification} of $S^d$ given by $\EH$. A stratification is \Defn{combinatorial} if it is a regular CW complex.

Let $\s$ be a combinatorial stratification of the sphere $S^d$, given by some fine extension of a $c$-arrangement~$\HA$. Let $\s^\ast$ be the dual block complex of $\s$. Define $\K(\HA,\s):=\RS^\ast(\s^\ast,S^d{\setminus} \HA)$. The regular CW complex $\K(\HA,\s)$ is the \Defn{complement complex} of $\HA$ induced by $\EH$. Complement complexes of $c$-arrangements in $\R^d$ are defined by restricting the complement complex of a spherical arrangement to a general position hemisphere, cf.\ Definition~\ref{def:affine}.

\begin{theorem}\label{thm:BZperfectc}
Any complement complex of any $c$-arrange\-ment $\HA$ in $S^d$ or~$\R^d$ admits a perfect Morse matching.
\end{theorem}

\begin{proof}
The proof of this theorem is analogous to the proof of Theorem~\ref{thm:BZperfect}.
\end{proof}

\noindent {\bf {\em Acknowledgements}} Special thanks to Bruno Benedetti, who communicated the problem solved in this paper (minimality of 2-arrangements) to me, and thanks also to Alex Suciu, who came up with it in the first place. Thanks to G\"unter Ziegler for helpful discussions and suggestions. Many thanks to the Institut Mittag-Leffler, Sweden, and the Hebrew University of Jerusalem, Israel, where the research leading to this article was conducted.
{\small
\newcommand{\etalchar}[1]{$^{#1}$}
\providecommand{\noopsort}[1]{}\def\cprime{$'$}
\providecommand{\bysame}{\leavevmode\hbox to3em{\hrulefill}\thinspace}
\providecommand{\MR}{\relax\ifhmode\unskip\space\fi MR }
\providecommand{\MRhref}[2]{%
  \href{http://www.ams.org/mathscinet-getitem?mr=#1}{#2}
}
\providecommand{\href}[2]{#2}

}

\end{document}